\documentclass[11pt,reqno]{amsart}
\usepackage{amsmath}
\usepackage{amssymb}
\usepackage{amsthm}
\usepackage{enumerate}
\usepackage[mathscr]{eucal}

\usepackage{mathtools}
\usepackage{microtype}

\usepackage{times}
\usepackage[english]{babel}

\usepackage{amsmath,amssymb,latexsym,textcomp,mathrsfs}
\usepackage[all]{xy}
\usepackage{graphicx}
\usepackage{bm,amsmath, amsthm, amssymb, amsfonts}
\usepackage{times}
\usepackage[english]{babel}

\setbox0=\hbox{$+$}
\newdimen\plusheight
\plusheight=\ht0
\def\+{\;\lower\plusheight\hbox{$+$}\;}  

\setbox0=\hbox{$-$}
\newdimen\minusheight
\minusheight=\ht0
\def\-{\;\lower\minusheight\hbox{$-$}\;}

\setbox0=\hbox{$\cdots$}
\newdimen\cdotsheight
\cdotsheight=\plusheight
\def\cds{\lower\cdotsheight\hbox{$\cdots$}}

\newcommand{\norm}[1]{\left\lVert#1\right\rVert}

\numberwithin{equation}{section}
\theoremstyle{plain}
\newtheorem{thm}{Theorem}[section]
\newtheorem{lem}[thm]{Lemma}

\newtheorem{prop}[thm]{Proposition}

\theoremstyle{definition}
\newtheorem{defn}{Definition}[section]

\newtheorem{exmp}{Example}[section]

\theoremstyle{remark}
\newtheorem{rem}{\bf{Remark}}
\newtheorem*{note}{\bf{Note}}
\numberwithin{equation}{section}

\begin{document}
\setcounter {page}{1}
\title{On $I$ and $I^*$-Cauchy conditions in $C^*$-algebra valued metric spaces}

\author{Amar Kumar Banerjee and Anirban Paul}

\address[A.K.Banerjee]{Department of Mathematics, The University of Burdwan, Golapbag, Burdwan-713104, West Bengal, India.}

\email{ akbanerjee1971@gmail.com, akbanerjee@math.buruniv.ac.in }

\address[A.Paul]{Department of Mathematics, The University of Burdwan, Golapbag, Burdwan-713104, West Bengal, India.}
         
\email{ paulanirban310@gmail.com, anirbanpaul310math@gmail.com}

\begin{abstract}  The idea of $C^*$-algebra valued metric spaces was given by Z. Ma et al \cite{111} in 2014. Here we have studied the ideas of $I$-Cauchy and $I^*$-Cauchy sequences and their properties in such spaces and also we give the idea of $C^*$-algebra valued normed spaces.
\end{abstract}
\maketitle
\author{}


\textbf {(2010) AMS subject classification :} 40A35, 54A20  \\
 
\section{Introduction}
The idea of statistical convergence of sequences of real number was introduced by Steinhaus \cite{q} and Fast \cite{f} as a generalizations of ordinary convergence of sequences of real numbers. Later many more works were done in this direction \cite{B, g, p}. In 2001, P. Kostryko et al \cite{h} introduced the idea of $I$-convergence of real sequences using the ideals of the set of natural numbers as a generalization of statistical convergence. Later in 2005, Lahiri and Das \cite{j} studied the same in a topological space and then many works were carried out in this direction \cite{2, 4, apurba, 5, paper, d, i}. In 2017, Banerjee and Mondal \cite{3} studied the same for double sequences in a topological space. Then the idea of $I$-Cauchy condition was studied by Dems \cite{dems}. Later $I^*$-Cauchy condition was studied by Nabiev et al \cite{m} and further investigated by Das et al \cite{c}.\\ 
One of the main direction in obtaining possible generalizations of these ideas in metric spaces is the study of same in a new type of spaces . In 2014, Z. Ma and L. Jiang \cite{111} introduced the concept of $C^*$-algebra valued metric spaces. The main idea of our paper deals using the set of positive elements of a unital $C^*$-algebra instead of the set of all real numbers. Obviously such spaces generalize the concepts of metric spaces as well as the cone metric spaces. In this paper we have studied on $I$ and $I^*$-Cauchy conditions in a $C^*$-algebra valued metric space and have discussed some results in this framework.\\
The idea of a cone normed space was introduced by D. Turkoglo et al \cite{s} which is a generalization of normed spaces. Based on the idea of $C^*$-algebra  we have introduced here the idea of $C^*$-algebra valued normed spaces which seems more general than the notion of cone normed spaces.\\
To begin with, we collect some definitions and basic facts on the theory of $I$-convergence in metric spaces and the theory of $C^*$-algebras, which will be needed in sequel.

\section{Preliminaries}
Recall that a Banach algebra $\mathbb{A}$ (over the field of complex numbers ) is said to be a $C^*$-algebra if there is an involution $*$ in $\mathbb{A}$ i.e, a mapping $ * : \mathbb{A}\mapsto \mathbb{A}$ satisfying $a^{**}=a$ for each $a \in \mathbb{A}$ such that, for all $a$, $b\in \mathbb{A}$ and $\lambda$, $\mu \in \mathbb{C}$ the following hold: \\
(i) $(\lambda a + \mu b)^* = \bar{\lambda} a + \bar{\mu} b;$\\
(ii) $(ab)^* = b^*a^*;$\\
(iii)$\norm{a^*a} = \norm{a}^2$\\
Note that, from (iii), it easily follows that $\norm{a}=\norm{a^*}$ for each $a\in \mathbb{A}$. Moreover if $\mathbb{A}$ contains an identity element $1_\mathbb{A}$ then $\mathbb{A}$ is said to be unital $C^*$-algebra. In the rest of the paper $\mathbb{A}$ will always be a unital $C^*$-algebra with unit $1_\mathbb{A}$ and the zero element $0_\mathbb{A}$. $\mathbb{A}_h$ will denote the set of all self-adjoint element $a$ (i.e, satisfying $a^*=a$). An element $a\in \mathbb{A}$ is called positive if $a\in \mathbb{A}_h$ and having spectrum $\sigma(a) \subset \mathbb{R}_+$, where $\mathbb{R}_+ = [0, \infty)$ and $\sigma(a)=\left\{\lambda\in \mathbb{R} : (\lambda 1_\mathbb{A} - a) \: \text{is non-invertible}\right\}$, the spectrum of $a$. We denote the set of all positive elements of $\mathbb{A}$ by $\mathbb{A^+}$. It is easy to see that $\mathbb{A^+}$ is a (closed) cone in the normed space $\mathbb{A}$ (see, Lemma 2.2.3\cite{murphy}) and thus a partial order $\preceq $ on $\mathbb{A}_h$ can be induced by $a \preceq b$ if and only if $b - a \in \mathbb{A}^+$. We now consider following simple results in $C^*$-algebra.\\
\begin{lem} (\cite{murphy})
(i) $\mathbb{A}^+= \left\{a^*a : a\in \mathbb{A}\right\}$;\\
(ii) if $a,\: b\in \mathbb{A}_h,\: a\preceq b,\: \text{and} \: c\in \mathbb{A}, \: \text{then}\: c^*ac \preceq  c^*bc$;\\
(iii) for all $a,\: b\in\mathbb{A}_h$ if $0_\mathbb{A} \preceq a \preceq b$ then $\norm{a} \leq \norm{b}$.\\
 \end{lem}
In the standard terminology used for cones in normed spaces, the property (iii) of the above lemma means that the cone $\mathbb{A}^+$ in $\mathbb{A}_h$ is normal with normal constant equal to $1$. The idea of $C^*$-algebra valued metric spaces has been given in \cite{111} analogously as the metric axioms (see \cite{4.1}) as follows:
\begin{defn}\cite{111}
Let $X$ be a non-empty set. A $C^*$-algebra valued metric on $X$ is a mapping $d : X\times X \mapsto \mathbb{A}$ satisfying the following conditions:\\
(I) $0_\mathbb{A} \preceq d(x, y)$ for all $x, y\in X$ and $d(x, y) = 0_\mathbb{A}$ if and only if $x = y$.\\
(II) $d(x, y) = d(y, x)$ for all $x, y\in X$.\\
(III) $d(x,y) \preceq d(x, z) + d(z, y)$, for all $x, y, z\in X$.\\

The triplet $(X, \mathbb{A}, d)$ is called a $C^*$-algebra valued metric space.
\end{defn}
\begin{rem}
The set of positive elements in a $C^*$-algebra forms a cone in $C^*$-algebra but the underlying vector space need not to be in general a real vector space.
\end{rem}
\begin{defn}\cite{111}
Let $(X, \mathbb{A}, d)$ be a $C^*$-algebra valued metric space. Suppose that $\left\{x_n\right\} \subset X$ and $x \in X$. If for any $\varepsilon > 0$ there is a $n_{0}$ such that, $\norm{d(x_n, x)}< \varepsilon$ whenever $n\geq n_{0}$, then $\left\{x_n\right\}$ is said to be convergent with respect to $\mathbb{A} $ and $\left\{x_n\right\}$ converges to $x$ and $x$ is the limit of $\left\{x_n\right\}$. We denote this by $\displaystyle{\lim_{n\to\infty}} x_n = x$.\\
If for any $\varepsilon > 0$ there is a $n_{0}$ such that, $\norm{d(x_n, x_m)} < \varepsilon$ whenever $m, n\geq n_{0}$, then $\left\{x_n\right\}$ is called Cauchy sequence with respect to $\mathbb{A}$.
\end{defn}
The following definitions and notion will be needed.\\

\begin{defn} \cite{t} Let $X\neq\phi$. A family $I \subset 2^X$ of subsets of $X$ is said to be an ideal of $X$ provided the following conditions holds:\\
\\
$(i)\: A, B \in I \Rightarrow A\cup B\in I$\\
$(ii) \: \:A\in I, B\subseteq A \Rightarrow B\in I$
\end{defn}
Note that $\phi\in I$ follows from the condition(ii). An ideal $I$ is called nontrivial if $I \neq \left\{\phi\right\}$ and $X\not\in I$. $I$ is said to be admissible if $\left\{x\right\}\in I$ for each $x\in X$.\\
Clearly if $I$ is a nontrivial ideal then the family of \sloppy sets $F(I)=\left\{M\subset X : \text{there exists}\: A\in I, M=X\setminus A\right\} $ is a filter in $X$. It is called the filter associated with the ideal $I$.

\begin{defn} \cite{h} Let $I\subset 2^\mathbb{N}$ be a nontrivial ideal of $\mathbb{N}$ and $(X,d)$ be a metric space. A sequence $\left\{x_n\right\}_{n\in \mathbb{N}}$ of elements of $X$ is said to be $I$-convergent to $x\in X$ if for each $\varepsilon>0$ the set $A(\varepsilon)=\left\{n\in \mathbb{N}:d(x_n,x)\geq \varepsilon\right\}$ belongs to $I$.
\end{defn}

\begin{defn} \cite{h}\label{1} An admissible ideal $I \subset 2^\mathbb{N}$ is said to satisfy the condition (AP) if for every countable family of mutual disjoint sets $\left\{A_1,A_2,A_3,\cdots\right\}$ belonging to $I$ there exists a countable family of sets $\left\{B_1,B_2,B_3,\cdots\right\}$ such that $A_j\Delta B_j$ is a finite set for each $j\in\mathbb{N}$ and $B=\displaystyle{\bigcup_{j=1}^{\infty}}B_j\in I$.
\end{defn}
Note that $B_j\in I$ for each $j\in \mathbb{N}$.\\

The concepts of $I^*$-convergence which is closely related to the $I$-convergence has been given in \cite{h} as follows:

\begin{defn} \cite{h} A sequence $\left\{x_n\right\}_{n\in\mathbb{N}}$ of elements of X is said to be $\text{I}^*$\text-convergent to $x\in X$ if and only if there exists a set $M=\left\{m_1<m_2 \cdots <m_k\cdots \right\}\in F(I)$ such that $\displaystyle{\lim_{k\to\infty}}d(x,x_{m_k})=0$.
\end{defn}

In \cite{h} it is seen that $I^*$-convergence implies $I$-convergence. If an admissible ideal $I$ has the property (AP) and $(X, d)$ is an arbitrary metric space, then for arbitrary sequence $\left\{x_n\right\}$ of elements of $X$, $I$-convergence implies $I^*$-convergence.\\
Throughout the paper $(X, \mathbb{A}, d)$ or simply $X$ denote the $C^*$-algebra valued metric space, $\mathbb{A}$ being the corresponding $C^*$-algebra, $\mathbb{N}$ will denote the set of all positive integers and $I$ stands for an admissible ideal of $\mathbb{N}$, $F(I)$ denote the filter associated with $I$ unless otherwise stated.
We first consider the following definitions.
\begin{defn}
Let$\left\{x_n\right\}$ be a sequence in $X$ and $x\in X$. If for any $\varepsilon >0$ the set $A(\varepsilon)=\left\{n\in N : \norm{d(x_n,x)} \geq \varepsilon\right\}\in I$, then $\left\{x_n\right\}$ is said to be $I$-convergent with respect to $\mathbb{A}$ and we write $I$-$\lim x_n = x$.
\end{defn}
\begin{defn}
Let $\left\{x_n\right\}$ be a sequence in $X$ and $x\in X$. If there exists a $M=\left\{m_1<m_2<\cdots<m_k<\cdots\right\}$ in $F(I)$ (i.e. $\mathbb{N}\setminus M\in I$) such that for any $\varepsilon>0$ there exists a positive integer $n_0$ satisfying $\norm{d(x_{m_k}, x)}<\varepsilon$ for all $k \geq n_{0}$, then $\left\{x_n\right\}$ is said to be $I^*$-convergent to $x$ with respect to $\mathbb{A}$ and we write $I^*$-$\lim x_n=x$.
\end{defn}
\section{$I$-Cauchy and $I^*$-Cauchy conditions in a $C^*$-algebra valued metric space.}\label{sec 3}
We now introduce the following definitions.
\begin{defn}
 A sequence $\left\{x_n\right\}$ in $X$ is called $I$-Cauchy sequence with respect to $\mathbb{A}$ if for each $\varepsilon>0$ there is a positive integer $m$ such that the set $A(\varepsilon)=\left\{ n\in \mathbb{N} : \norm{d(x_{n_{0}}, x_n} \geq \varepsilon\right\} \in I$.
\end{defn}
\begin{defn}
A sequence $x=\left\{x_n\right\}$ in $X$ is called an $I^*$-Cauchy sequence with respect to $\mathbb{A}$, if there exists a set $M=\left\{m_1<m_2<\cdots<m_k<\cdots\right\}\in F(I)$ such that the sub sequence $x_k=\left\{x_{m_k}\right\}$ is a Cauchy sequence in $(X, \mathbb{A}, d)$. i.e. for any $\varepsilon >0$ there is a $m$ such that $\norm{d(x_{m_k}, x_{m_p})} < \varepsilon $ for all $k, p >m$
\end{defn}
\begin{exmp}\label{ex1}
Let $X=\mathbb{R}, \:\mathbb{A}=M_2(\mathbb{R})$, the space of all $2\times 2$ real matrices. Define $d(x, y)= \text{diag} (|x -y |, \alpha| x- y|), \alpha \geq 0$ is a constant. Now $\mathbb{A}$ can be identified with $B(\mathbb{R}^2)$, the set of all bounded linear maps from the Hilbert space $\mathbb{R}^2$ to $\mathbb{R}^2$, which is obviously a unital Banach-algebra with pointwise-defined operations for addition: $(T_1 + T_2)(x)= T_1(x) +T_2(x)$ and the scalar multiplication : $(\lambda T)(x)=\lambda T(x)$, the multiplication of operators is given by $(T_1, T_2)\mapsto T_1 \circ T_2$, and the norm is given by the operator norm $\norm{T}=\displaystyle{\sup_{\norm{x}=1, x\in \mathbb{R}^2}} \norm{T(x)}$. $\mathbb{A}$ becomes a $C^*$-algebra by taking Hilbert adjoint operator as the the involution map $T \mapsto T^*$. Now we see that $d(x, y)= T_{xy}$ where $T_{xy}(x', y')=( x'|x -y|, \alpha y'|x -y|)$. Now consider the sequence $x_n=\frac{1}{n}$. Let $\varepsilon >0$ be arbitrary. Then we can choose a $n_0\in \mathbb{N}$ such that $\frac{1}{n_0} < \varepsilon$. Now, $d(x_n, x_{n_{0}})= T_{x_nx_{n_{0}}}$, where $T_{x_n x_{n_0}}(x, y) = (x|x_n - x_{n_0}|, \alpha y|x_n - x_{n_0}|)= (\frac{|n_0 - n|}{n_0n}x, \alpha(\frac{|n_0 -n|}{n_0n})y)$. So, $\norm{d(x_n, x_{n_0})}=\norm
{T_{x_n x_{n_0}}}=\displaystyle{\sup_{\norm{x}=1, x\in \mathbb{R}^2}}\norm{T_{x_n x_{n_0}}(x)}= \begin{cases}\alpha\frac{|n_0- n|}{n_0n}\: \text{if} \:\alpha >1  &\\ \frac{|n_0 -n |}{n_0n}\: \text{if}\: 0\leq\alpha < 1 \end{cases}$. Now $A(\varepsilon)=\left\{n\in \mathbb{N} : \frac{|n_0 - n |}{n_0n}\geq \varepsilon\right\}$ or $\left\{n\in \mathbb{N} : \alpha(\frac{|n_0 -n |}{n_0n})\right\}$ both of the set is finite set. Hence belongs to $I$, as $I$ is an admissible ideal. This shows that $x_n=\frac{1}{n}$ is $I$-Cauchy sequence with respect to $\mathbb{A}$ in $X$.
\end{exmp}
\begin{note}
It may happens that there exists an $C^*$-algebra $\mathbb{A}$ with respect to which a sequence is $I$-Cauchy but there may exists another $C^*$-algebra $\mathbb{B}$ with respect to which the same sequence may not be $I$-Cauchy. The following example is such one.
\end{note}
\begin{exmp}
We have seen in the example \ref{ex1} that the sequence $\left\{\frac{1}{n}\right\}$ is $I$-Cauchy with respect to $\mathbb{A}$ in $(X, \mathbb{A}, d)$. Now let us consider the collection $\mathbb{B}=\ell^\infty(S)$, the set of all bounded complex-valued functions on $S$, where $S=[a,b]$. Then $\mathbb{B}$ is unital Banach algebra where the operations are defined pointwise: $(f + g)(x)= f(x) + g(x)$, $(fg)(x)= f(x)g(x)$, $(\lambda f)(x)=\lambda f(x)$ and the norm is the sup-norm  ${\norm f}_{\infty}= \displaystyle{\sup_{x\in S}}|f(x)|$. With the involution $f \mapsto f^*$ defined by $f^*=\bar{f}$, where `$\bar{f}$' denotes the complex conjugation $\lambda \mapsto \bar{\lambda}$ of the function $f$, \:$\mathbb{B}$ becomes a $C^*$-algebra. Let $d(x,y)=\begin{cases} \frac{1}{|x - y|}f \: \text{if}\: x\neq y &\\ 0 \: \text{if}\: x=y \end{cases}$, where $f\in \mathbb{B}$ such that $\norm{f} >1$ and $f$ is kept fixed. Let $\varepsilon>0$ be given. Then we can choose a $n_0\in \mathbb{N}$ such that $\frac{1}{n_0}< \varepsilon$. Then $d(x_n, x_{n_0})=\frac{1}{|\frac{1}{n} - \frac{1}{n_0}|}f =|\frac{n_0 n}{n_0 - n}|f$. So $\norm{d(x_n, x_{n_0})}=|\frac{n_0 n}{n_0 - n}|\norm{f}$. Therefore $A(\varepsilon)=\left\{n\in \mathbb{N} : \norm{d(x_n, x_{n_0})}\geq \varepsilon\right\}$ excludes only finite number of $n\in \mathbb{N}$. Now as $I$ is an admissible ideal, so $A(\varepsilon)\notin I$. Therefore $\left\{x_n\right\}$ is not $I$-Cauchy with respect to $\mathbb{B}$.
\end{exmp}
Let $\varepsilon > 0$ and $\left\{x_n\right\}$ be a sequence in $(X,\mathbb{A}, d)$. We denote $E_k(\varepsilon)=\left\{n\in \mathbb{N}: \norm{d(x_n, x_k)}\geq \varepsilon\right\}$, $k\in\mathbb{N}$. Then we have the following proposition.
\begin{prop}\cite{dems}\label{prop1}
For a sequence $\left\{x_n\right\}$ of points in $(X, \mathbb{A}, d)$ the following are equivalent.\\
(1)$\left\{x_n\right\}_{n\in\mathbb{N}}$ is an $I$-Cauchy sequence with respect to $\mathbb{A}$.\\
(2)($\forall \varepsilon>0$) ($\exists D\in I$) ($\forall m,n\not\in D$) $\norm{d(x_m,x_n)}<\varepsilon$.\\
(3)($\forall \varepsilon>0$)   $\left\{k\in\mathbb{N} : E_k(\varepsilon)\not\in I\right\}\in I$
\end{prop}
\begin{proof}
(1)$\Rightarrow$(2) Let $\varepsilon>0$. Then for $\frac{\varepsilon}{3}$ there exists a $k\in \mathbb{N}$ such that $D=\left\{n\in \mathbb{N} : \norm{d(x_n, x_k)}\geq \frac{\varepsilon} {3}\right\}\in I$. Now for any $n, m\notin D$ we have $\norm{d(x_n, x_k)}< \frac{\varepsilon}{3}$ and $\norm{d(x_m,x_k)}<\frac{\varepsilon}{3}$. Hence $\norm{d(x_n, x_m)}\leq \norm{d(x_n,x_k)} +\norm{d(x_k,x_m)} < \frac{2\varepsilon}{3} <\varepsilon$.\\
(2)$\Rightarrow$(3) Let $\varepsilon >0$. Then as in ($2$) there exists $D=E_{k}(\frac{\varepsilon}{3})=\left\{n\in \mathbb{N} : \norm{d(x_n, x_k)}\geq \frac{\varepsilon}{3}\right\}\in I$ such that if $m, n\notin D$ then $ \norm{d(x_n, x_m)}< \varepsilon$. We have to prove that $\left\{k\in \mathbb{N} : E_k(\varepsilon)\notin I\right\}\in I$. We show $\left\{k\in \mathbb{N} : E_k(\varepsilon)\notin I\right\}\subset D$. If possible let $p\in \left\{k\in \mathbb{N} : E_k(\varepsilon)\notin I\right\}$ such that $p\notin D$. Then $E_p(\varepsilon)\notin I$. So $E_p(\varepsilon) \not\subset D$. Choose $r \in E_p(\varepsilon)\setminus D$. As $r\in E_p(\varepsilon)$, so $\norm{d(x_r, x_p)}\geq \varepsilon$. But as $r, p\notin D$ by the condition $\norm{d(x_r, x_p)}< \varepsilon$, a contradiction.\\
(3)$\Rightarrow$(1) From (3) we have $\left\{k\in\mathbb{N} : E_k(\varepsilon)\in I\right\}\in F(I)$ \: $\forall\varepsilon> 0$. Therefore $\left\{k\in\mathbb{N} : E_k(\varepsilon)\in I\right\}\neq \phi$ \: for all $\varepsilon >0$. Hence there exists $k\in\mathbb{N}$ such that $E_k(\varepsilon)\in I$. Hence $(1)$ follows.
\end{proof}
\begin{thm}\label{cauchypro} 
Let $I$ be an arbitrary admissible ideal. Then$I\text{-}\displaystyle{\lim_{n\to\infty}}x_n=\xi$ implies that $\left\{x_n\right\}$ is an $I$-Cauchy sequence with respect to $\mathbb{A}$ in $(X,\mathbb{A},d)$.
\end{thm}
\begin{proof}
 Let $I\text{-}\displaystyle{\lim_{n\to\infty}}x_n=\xi$. Then for each $\varepsilon >0$, we have$A(\varepsilon)=\left\{n\in\mathbb{N}:\norm{d(x_n,\xi)}\geq\varepsilon\right\}\in I$. Since $I$ is an admissible ideal there exists an $n_0\in \mathbb{N}$ such that $n_0\notin A(\varepsilon)$. Let $B(\varepsilon)=\left\{n\in\mathbb{N} : \norm{d(x_n, x_{n_0})}\geq 2\varepsilon\right\}$. Now $\norm{d(x_n,x_{n_o})} \leq \norm{d(x_n, \xi)} + \norm{d(x_{n_0},\xi)}$ ( since $0_\mathbb{A}\preceq a\preceq b \Rightarrow \norm a \leq \norm b$ ). So, if $n\in B(\varepsilon)$ then we get $2\varepsilon\leq \norm{d(x_n,x_{n_0})}\leq \norm{d(x_n, \xi)} +\norm{d(x_{n_0}, \xi)}$. Therefore $2\varepsilon\leq \norm{d(x_n,x_{n_0})} < \varepsilon + \norm{d(x_n, \xi)}$ ( since $n_0\notin A(\varepsilon)$). So $\norm{d(x_n, \xi)}>\varepsilon $. This shows that $n\in A(\varepsilon)$. Thus we see that $B(\varepsilon) \subset A(\varepsilon)$. Also $A(\varepsilon)\in I$ for  each $\varepsilon >0$. This gives $B(\varepsilon)\in I$  i.e. $\left\{x_n\right\}$ is $I$-Cauchy sequence. Hence the proof is complete.   
\end{proof}
\begin{thm}\label{3.3}
Let $I$ be an admissible ideal. If $x=\left\{x_n\right\}$ is $I^*$-Cauchy sequence with respect to $\mathbb{A}$ in $X$  then it is also $I$-Cauchy sequence with respect to $\mathbb{A}$.
\end{thm}
\begin{proof}
Let $x=\left\{x_n\right\}$ be an $I^*$-Cauchy sequence in $X$ with respect to $\mathbb{A}$. Then by definition there exists a set $M=\left\{m_1<m_2<\cdots<m_k<\cdots\right\}\in F(I)$ such that for any $\varepsilon >0$ there exists $k_0=k_0(\varepsilon)$ such that $\norm{d(x_{m_k}, x_{m_p})}<\varepsilon$ for all $k, p > k_{0}$. Let $n_0=m_{k_{0} +1}$. Then for every $\varepsilon >0$ we have $\norm{d(x_{m_k}, x_{n_0})}<\varepsilon$ for all $k > k_0$. Now the set $H=\mathbb{N}\setminus M \in I$ and\begin{equation*}
A(\varepsilon)=\left\{n\in\mathbb{N} : \norm{d(x_n, x_{n_0})}\geq \varepsilon\right\}\subset H \cup\left\{m_1<m_2<\cdots<m_{k_0}\right\}.\end{equation*} The set $H \cup\left\{m_1<m_2<\cdots<m_{k_0}\right\}\in I$. Therefore, $A(\varepsilon)\in I$ and hence $\left\{x_n\right\}$ is $I$-Cauchy with respect to $\mathbb{A}$.
\end{proof}
\begin{note}
$I^*$-convergent sequence in $X$ with respect to $\mathbb{A}$ always implies that it is also $I$-Cauchy with respect to $\mathbb{A}$. For, let $\left\{x_n\right\}$ be a sequence in $X$ which is $I^*$-convergent with respect to $\mathbb{A}$ to $x\in X$. Then there exists $M=\left\{m_1<m_2<\cdots<m_k<\cdots\right\}\in F(I)$ ( i.e. $\mathbb{N}\setminus M \in I$ ) such that for any $\varepsilon >0$, $\norm{d(x_{m_k}, x)}<\varepsilon/2$ for all $k>N=N(\varepsilon/2)$. Now $\norm{d(x_{m_k}, x_{m_p})}\leq \norm{d(x_{m_k},x)} + \norm{d(x_{m_p}, x)} \leq \norm{d(x_{m_k}, x)} + \norm{d(x_{m_p},x)} <\varepsilon/2 +\varepsilon/2=\varepsilon $ for all $k, p>N$. So $\left\{x_n\right\}$ is $I^*$-Cauchy in $(X,\mathbb{A}, d)$. Now as $I^*$-Cauchy sequence is $I$-Cauchy sequence with respect to $\mathbb{A}$, so a $I^*$-convergent sequence always implies that it is also $I$-Cauchy with respect to $\mathbb{A}$.
\end{note}
In general $I$-Cauchy condition with respect to $\mathbb{A}$ does not imply $I^*$-Cauchy condition with respect to $\mathbb{A}$, as shown in the following example.
\begin{exmp}
Let $X=\mathbb{R}$ and $\mathbb{A}=\ell^\infty(S)$, the set of all bounded complex-valued functions on $S$, where $S=[a,b]$. Then $\mathbb{A}$ is unital Banach algebra where the operations are defined pointwise: $(f + g)(x)= f(x) + g(x)$, $(fg)(x)= f(x)g(x)$, $(\lambda f)(x)=\lambda f(x)$ and the norm is given by  the sup-norm i.e, ${\norm f}_{\infty}= \displaystyle{\sup_{x\in S}}|f(x)|$. With the involution $f \mapsto f^*$ defined by $f^*=\bar{f}$, where `$\bar{f}$' denotes the complex conjugation $\lambda \mapsto \bar{\lambda}$ of the function $f$, \:$\mathbb{A}$ becomes a $C^*$-algebra. Now let $\mathbb{N}=\displaystyle{\bigcup_{j\in\mathbb{N}}}\Delta_j$ be a decomposition of $\mathbb{N}$ such that each $\Delta_j=\left\{2^{j -1}( 2s -1) : s=1, 2, 3,\cdots\right\}$. Then $\Delta_i \cap\Delta_j=\phi$ for $i\neq j$. Let $I$ be the class of all those subsets of $\mathbb{N}$ that can intersects with only a finite numbers of $\Delta_i$'s. Then clearly $I$ is a non-trivial admissible ideal of $\mathbb{N}$. Let $x_n=\frac{1}{j}$ if $n\in \Delta_j$. Let $d(x,y)=|x - y|f$, where $f\in C\mathbb{A}$ with $\norm{f} > 1$ and $f$ is kept fixed. Let $\varepsilon >0$ be given. Now there is a $N=N(\varepsilon)\in\mathbb{N}$ such that $\frac{1}{N}< \frac{\varepsilon}{2 s}$, where $s=\displaystyle{\sup_{x\in \mathbb{R}}}|f(x)|$. Then $\norm{d(\frac{1}{n},\frac{1}{m})}=|\frac{1}{n} - \frac{1}{m}|s \leq (\frac{1}{n}+ \frac{1}{m})s < s(\frac{\varepsilon}{2s} + \frac{\varepsilon}{2s})=\varepsilon$ for all $n, m \geq N$. Now $B=\Delta_1\cup\Delta_2\cup\cdots\cup\Delta_N\in I$ and clearly $m, n\notin B$ implies $\norm{d(x_m,x_n)}<\varepsilon$. Hence $\left\{x_n\right\}$ $I$-Cauchy in $(X,\mathbb{A}, d)$ by the preposition \ref{prop1}.
Next we shall show that $\left\{x_n\right\}$ is not $I^*$-Cauchy with respect to $\mathbb{A}$. If possible assume that $\left\{x_n\right\}$ is $I^*$-Cauchy. Then there is a $M\in F(I)$ such that $\left\{x_n\right\}_{n\in M}$ is Cauchy sequence with respect to $\mathbb{A}$. Since $\mathbb{N}\setminus M\in I$, so there exists a $l\in \mathbb{N}$ such that $\mathbb{N}\setminus M\subset \Delta_1\cup\Delta_2\cup\cdots\Delta_l$ but then $\Delta_i\subset M$ for all $i>l$. In particular $\Delta_{l +1}, \Delta_{l+2} \subset M$. Now from the construction of $\Delta_i$'s it follows that given any $k\in\mathbb{N}$ there are $m\in\Delta_{l+1}$ and $n\in \Delta_{l+2}$ such that $m,n \geq k$. Hence $\norm{d(x_m, x_n)}=\frac{1}{(l+1)(l+2)}s$ where $s=\norm f$. Choose $\varepsilon_0=\frac{s}{3(l+1)(l+2)}$. Therefore there is no $k\in \mathbb{N}$ such that whenever $m, n \in \mathbb{N}$ with $m, n\geq k$ then $\norm{d(x_n, x_m)}<\varepsilon_0$ holds. This contradicts the fact that $\left\{x_n\right\}_{n\in M}$ is Cauchy with respect to $\mathbb{A}$. 
\end{exmp}
However the converse of Theorem \ref{3.3} holds if the condition (AP) hold for an admissible ideal $I$.
To prove that we need the following lemma. 

\begin{lem}\cite{dems}\label{lem} Let $\left\{P_i\right\}_{i=1}^{\infty}$ be a countable collection of subsets of $\mathbb{N}$ such that $P_i\in F(I)$ for each $i$, where $F(I)$ is a filter associated with an admissible ideal $I$ with the property (AP). Then there exists a set $P\subset \mathbb{N}$ such that $P\in F(I)$ and the set $P\setminus P_i$ is finite for all $i$.
\end{lem}
\begin{thm}\label{last}
If $I$ be an admissible ideal with the property (AP) and if $\left\{x_n\right\}$ is a $I$-Cauchy sequence in $X$ with  respect to a $C^*$-algebra $\mathbb{A}$ then it is $I^*$-Cauchy with respect to a $C^*$-algebra $\mathbb{A}$
\end{thm}
\begin{proof}
Let $x=\left\{x_n\right\}$ in $X$ be an $I$-Cauchy sequence with respect to $\mathbb{A}$ and let $\varepsilon > 0$ be arbitrary. Then by definition, there exists a positive integer $n_0$ such that $A(\varepsilon)=\left\{n\in\mathbb{N}        : \norm{d(x_n,x_{n_0})}\geq \varepsilon\right\}\in I$. So in particular for each $k\in \mathbb{N}$, there exists $m_k$ such that $A(\frac{1}{k})=\left\{n\in \mathbb{N} :\norm{d(x_n, x_{m_k})}\geq \frac{1}{k}\right\}\in I$. Let $B_k=\left\{n\in\mathbb{N} :\norm{d(x_n,x_{m_k})}<\frac{1}{k}\right\}$, where $k=1,2,\cdots$. It is clear that $B_k\in F(I)$ for $k=1,2,\cdots $. Since $I$ has the property (AP), then by lemma (\ref{lem}) there exists a set $P\subset \mathbb{N}$ such that $P\in F(I)$, and $P\setminus B_k$ is finite for all $k$. Now let $j\in \mathbb{N}$ be such that $j>\frac{2}{\varepsilon}$. As $P\setminus B_j$ is a finite set there exists $k$ such that $m, n\in P $ implies $m\in B_j$ and $n\in B_j$ for all $m,n> k$. Therefore, $\norm{d(x_n,x_{m_j})}<\frac{1}{j}$ and $\norm{d(x_m,x_{m_j})}<\frac{1}{j}$ for all $m,n >k$. Hence it follows that $\norm{d(x_n,x_m)}\leq \norm{d(x_n,x_{m_j})} + \norm{d(x_m,x_{m_j}) }<\frac{2}{j}<\varepsilon$ for $m,n>k$. Thus, there exists $P\in F(I)$ such that for any $\varepsilon>0$ there exists $k$ satisfying $m,n>k$, $m,n\in P\in F(I)\Rightarrow \norm{d(x_n,x_m)}<\varepsilon$. This shows that the sequence $\left\{x_n\right\}$ in $X$ is an $I^*$-Cauchy sequence with respect to $\mathbb{A}$.
\end{proof}
\section{$C^*$ algebra valued normed space}
We will now introduce the concepts of $C^*$-algebra valued normed space.
\begin{defn}
Let $X$ be a real vector space. Suppose that the mapping $\norm{.}_{\mathbb{A}} : X \mapsto \mathbb{A}$ be such that\\
(i) $\norm{x}_{\mathbb{A}}\geq 0_{\mathbb{A}}$ for all $x\in X$ and $\norm{x}_{\mathbb{A}}=0_{\mathbb{A}}$ if and only if $x=\theta_X$, where $\theta_X$ is the zero element of $X$.\\
(ii) $\norm{\alpha x}_{\mathbb{A}} = |\alpha|\norm{x}_{\mathbb{A}}$, $x\in X$, $\alpha\in \mathbb{R}$.\\
(iii) $\norm{ x+ y}_{\mathbb{A}}\leq \norm{x}_{\mathbb{A}} +\norm{y}_{\mathbb{A}}$ for all $x, y\in X$.\\
Then $\norm{.}_{\mathbb{A}}$ is called a $C^*$-algebra valued norm on $X$ and the pair $(\norm{.}_{\mathbb{A}}, X)$ is called $C^*$-algebra valued normed space.
\end{defn}
It is easy to show that every normed space is a $C^*$-algebra valued normed space by putting $\mathbb{A}=\mathbb{R}$ with the involution map $a\mapsto a^*$ taken to be $a^*= a$ for all $a\in \mathbb{A}$.
\begin{exmp}
Let $X=\mathbb{R}$, $\mathbb{A}=M_2(\mathbb{R})$, the space of all $2\times 2$ real matrices. Now $\mathbb{A}$ can be identified with $B(\mathbb{R}^2)$, the set of all bounded linear maps from the Hilbert space $\mathbb{R}^2$ to $\mathbb{R}^2$, which is obviously a unital Banach-algebra with pointwise-defined operations for addition : $(T_1 + T_2)(x)= T_1(x) +T_2(x)$ and the scalar multiplication : $(\lambda T)(x)=\lambda T(x)$, the multiplication of operators is given by $(T_1, T_2)\mapsto T_1 \circ T_2$, and the norm is given by the operator norm $\norm{T}=\displaystyle{\sup_{\norm{x}=1, x\in \mathbb{R}^2}} \norm{T(x)}$. $\mathbb{A}$ becomes a $C^*$-algebra by taking the involution map $T \mapsto T^*$ to be the Hilbert adjoint operator. Let $(X, \norm{.})$ be a normed space with usual norm on $\mathbb{R}$. Let us define a map $\norm{.}_{\mathbb{A}}: X \mapsto \mathbb{A}$ by $\norm{x}_{\mathbb{A}}=\norm{x} \begin{pmatrix}
a &0\\
0 &b
\end{pmatrix}$, where $a, b$ is being kept fixed. Then $(X, \norm{.}_{\mathbb{A}})$ is a $C^*$-algebra valued normed space.
\end{exmp}
\begin{rem}
Let $(X,\norm{.}_{\mathbb{A}})$ be a $C^*$-algebra valued normed space. Set $D(x, y)=\norm{x -y}_{\mathbb{A}}$. Then clearly $(X,D)$ becomes a $C^*$-algebra valued metric space. $D$ is then called ``the $C^*$-algebra valued metric induced by the $C^*$-algebra valued norm $\norm{.}_{\mathbb{A}}$"
\end{rem}

\begin{thm}\label{homo}
The $C^*$-algebra valued metric $D$ induced by a $C^*$-algebra valued norm satisfies:\\
(i) $D(x +a , y+ a)=D(x ,y)$\\
(ii)$D(\alpha x, \alpha y)=|\alpha| D(x, y)$, where $x, y, a\in X$ and $\alpha \in \mathbb{R}$.
\end{thm}
\begin{proof}
We have $D(x +a, y+a)=\norm{(x + a) - (y + a)}_{\mathbb{A}}=\norm{x - y}_{\mathbb{A}}=D(x, y)$ and $D(\alpha x, \alpha y)=\norm{\alpha x - \alpha y}_{\mathbb{A}}=|\alpha|\norm{x - y}_{\mathbb{A}}=|\alpha| D(x, y)$.
\end{proof}
The following example is given to show that $C^*$-algebra valued matrices do not necessarily produce $C^*$-algebra valued norm.
\begin{exmp}
Let $X=\mathbb{R}$, $\mathbb{A}=M_2(\mathbb{R})$. We can identify $\mathbb{A}$ with $B(\mathbb{R}^2)$, the set of all bounded linear maps from the Hilbert space $\mathbb{R}^2$ to $\mathbb{R}^2$, which is obviously a unital Banach-algebra with pointwise-defined operations for addition : $(T_1 + T_2)(x)= T_1(x) +T_2(x)$ and the scalar multiplication : $(\lambda T)(x)=\lambda T(x)$, the multiplication of operators is given by $(T_1, T_2)\mapsto T_1 \circ T_2$, and the norm is given by the operator norm $\norm{T}=\displaystyle{\sup_{\norm{x}=1, x\in \mathbb{R}^2}} \norm{T(x)}$. By taking the Hilbert adjoint operator as the involution map $T \mapsto T^*$, \:$\mathbb{A}$ becomes a $C^*$-algebra. Let $d(x, y)=\begin{cases} \begin{pmatrix}
1 &0\\
0&1
\end{pmatrix},\: \text{if} \:x\neq y\\
0, \:\:\:\text{if} \: x=y
\end{cases}$. Then $d$ is a $C^*$-algebra valued metric space. Clearly $d$ does not satisfies the property as in the theorem \ref{homo}.
\end{exmp}
Convergence in a $C^*$-algebra valued normed space is described by the $C^*$-algebra valued metric induced by the $C^*$-algebra valued norm.
\begin{defn}
A sequence $\left\{x_n\right\}\in X$ said to converge an element $x\in X$, if for any $\varepsilon>0$ there exist a $n_0\in \mathbb{N}$ such that $\norm{D(x_n, x)}=\norm{\norm{x_n - x}_{\mathbb{A}}}<\varepsilon$ for all $n> n_0$.
\end{defn}
Hence a sequence $x_n\rightarrow x$ if and only if $\norm{D(x_n, x)}\rightarrow 0$ as $n\rightarrow\infty$.
\begin{rem}
All the results of section \ref{sec 3} viz  Proposition \ref{prop1}, Theorem \ref{cauchypro}, Theorem \ref{3.3}, Theorem \ref{last} also hold if we take the $C^*$-algebra valued normed spaces $( X, \norm{.}_{\mathbb{A}})$ instead of $C^*$-algebra valued metric spaces $( X, \mathbb{A}, d)$
\end{rem}

\end{document}